\documentclass[11pt]{article}

\usepackage{amsmath,amssymb,amsfonts,amsthm}
\usepackage{amssymb}
\usepackage{graphics,graphicx}                 
\usepackage{tikz}
\usetikzlibrary{automata}
\usetikzlibrary{arrows}
\tikzset{every loop/.style={min distance=10mm,looseness=10}}
\tikzset{every state/.style={minimum size=2mm}}

\newtheorem{theorem}{Theorem}
\newtheorem{lemma}[theorem]{Lemma}
\newtheorem{remark}[theorem]{Remark}

\title{On semi-transitive orientability of Kneser graphs and their complements}

\author{Sergey Kitaev\footnote{Department of Computer and Information Sciences, University of Strathclyde, 26 Richmond Street, Glasgow G1, 1XH, United Kingdom. 
{\bf Email:} sergey.kitaev@cis.strath.ac.uk.}\ \ and Akira Saito\footnote{Department of Information Science, Nihon University, Sakurajosui 3-25-40
Setagaya-Ku Tokyo 156--8550, Japan. {\bf Email:} asaito@chs.nihon-u.ac.jp.}}

\begin{document}

\maketitle

\begin{abstract}
An orientation of a graph is semi-transitive if it is acyclic, and for any directed path $v_0\rightarrow v_1\rightarrow \cdots\rightarrow v_k$ either there is no edge between $v_0$ and $v_k$, or $v_i\rightarrow v_j$ is an edge for all $0\leq i<j\leq k$. An undirected graph is semi-transitive if it admits a semi-transitive orientation. Semi-transitive graphs include several important classes of graphs such as 3-colorable graphs, comparability graphs, and circle graphs, and they are precisely the class of word-representable graphs studied extensively in the literature.  

In this paper, we study semi-transitive orientability of the celebrated Kneser graph $K(n,k)$, which is the graph whose vertices correspond to the $k$-element subsets of a set of $n$ elements, and where two vertices are adjacent if and only if the two corresponding sets are disjoint. We show that for $n\geq 15k-24$, $K(n,k)$ is not semi-transitive, while for $k\leq n\leq 2k+1$, $K(n,k)$ is semi-transitive. Also, we show computationally that a subgraph $S$ on 16 vertices and 36 edges of $K(8,3)$, and thus $K(8,3)$ itself on 56 vertices and 280 edges, is non-semi-transitive. $S$ and $K(8,3)$ are the first explicit examples of triangle-free non-semi-transitive graphs, whose existence was established via Erd\H{o}s' theorem by Halld\'{o}rsson et al. in 2011. Moreover, we show that the complement graph $\overline{K(n,k)}$ of $K(n,k)$ is semi-transitive if and only if $n\geq 2k$.
\end{abstract}

\section{Introduction}
An orientation of a graph is {\em semi-transitive} if it is acyclic, and for any directed path $v_0\rightarrow v_1\rightarrow \cdots \rightarrow v_k$ either there is no edge between $v_0$ and $v_k$, or $v_i\rightarrow v_j$ is an edge for all $0\leq i<j\leq k$. The notion of a semi-transitive orientation generalizes that of a {\em transitive orientation}. An undirected graph is {\em semi-transitive} if it admits a semi-transitive orientation. Not all graphs are semi-transitive, and the minimum (by the number of vertices) non-semi-transitive graph is the {\em wheel graph} $W_5$ on 6 vertices. Note that any complete graph can be oriented transitively, and thus semi-transitively.

A {\em shortcut} $C$ in a directed acyclic graph is an induced subgraph on vertices  $\{v_0,v_1,\ldots, v_k\}$ for $k\geq 4$ such that $v_0\rightarrow v_1\rightarrow \cdots \rightarrow v_k$ is a directed path, $v_0\rightarrow v_k$ is an edge, and there exist $0\leq i<j\leq k$ such that there is no edge between $v_i$ and $v_j$. Thus, $C$ has no directed cycles and it is non-transitive, and an orientation is semi-transitive if and only if it is acyclic and shortcut-free. The edge $v_0\rightarrow v_k$ in $C$ is called the {\em shortcutting edge}, and the path $v_0\rightarrow v_1\rightarrow \cdots \rightarrow v_k$ is the {\em long path} in $C$.

The notion of a semi-transitive orientation was introduced by Halld\'{o}rsson, Kitaev and Pyatkin \cite{HKP11} in 2011 as a powerful tool to study {\em word-representable graphs} defined in Section~\ref{sec2} via alternation of letters in words and studied extensively in the recent years (see  \cite{K17,KL15} and references therein). The class of semi-transitive graphs is precisely the hereditary class of word-representable graphs. The roots of the theory of word-representable graphs, i.e.\ semi-transitive graphs, are in the study of the celebrated {\em Perkins semigroup} in \cite{KS08}, which has played a central role in semigroup theory since 1960, particularly as a source of examples and counterexamples. However, the significance of the class of semi-transitive graphs is in the fact that it includes several important classes of graphs such as  $3$-colorable graphs, comparability graphs and circle graphs \cite{K17,KL15}. Note that the user-friendly software \cite{G} by Glen is of special importance for the development of the area, and we use it in this paper to show a particular result. 

For any two integers $k\geq 1$ and $n\geq k$, the  Kneser graph $K(n,k)$ is the graph whose vertices correspond to the $k$-element subsets of a set $[n]:=\{1,2,\ldots,n\}$, where two vertices are adjacent if and only if the two corresponding sets are disjoint\footnote{Often, when defining the Kneser graph $K(n,k)$, the assumption is that $n\geq 2k+1$ to avoid dealing with null graphs, which are meaningful in our context, since we also deal with $\overline{K(n,k)}$.}. In particular, $K(5,2)$ is isomorphic to the celebrated {\em Petersen graph}. When writing down subsets, we omit brackets and commas. Thus, for example, the subset $\{1,4,6,7\}$ is recorded by us as 1467, while the subset $\{2,9,10\}$ as 29(10), etc. A Kneser graph is vertex transitive and edge transitive, and these graphs are named after Martin Kneser, who first investigated them in 1955. There is a long line of research dedicated to Kneser graphs; see the recent \cite{MNW} and references therein. 

In this paper we study semi-transitive orientability of Kneser graphs and their complements. Our main results can be summarized as follows:
\begin{itemize}
\item For $n\geq 15k-24$, $K(n,k)$ is not semi-transitive (see Theorem~\ref{Knes-main-thm}), while for $k\leq n\leq 2k+1$ $K(n,k)$ is semi-transitive (see Theorem~\ref{3-col-Kneser}). Moreover, it is shown by the software \cite{G} that $K(8,3)$ is not semi-transitive (see the discussion at the end of Section~\ref{sec3}). A certain subgraph $S$ of $K(8,3)$ presented in Figure~\ref{K83-fig} and $K(8,3)$ itself are the first explicit examples of triangle-free non-semi-transitive graphs, whose existence was established via Erd\H{o}s' theorem by Halld\'{o}rsson, Kitaev and Pyatkin in \cite{HKP11}.
\item The complement graph $\overline{K(n,k)}$ of $K(n,k)$ is semi-transitive if and only if $n\geq 2k$  (see Theorem~\ref{complement-main}). 
\end{itemize}

\section{Known results to be used in the paper}\label{sec2}

We begin with the following well known result. 

\begin{theorem}\label{long-directed-path} Let $G$ be an acyclically oriented graph with chromatic number $m$. Then, $G$ contains a directed path of length at least $m-1$. \end{theorem}

\begin{proof}  Any acyclic orientation contains a source, i.e.\ a vertex with no in-coming edges. Consider all sources and colour them in colour 1. Remove all sources along with the edges connected to them, and colour the sources in the obtained acyclic graph in colour 2. Proceed in this way. If the longest path in $G$ is of length at most $m-2$, then it is possible to colour $G$ in at most $m-1$ colours contradicting the chromatic number of $G$ being $m$. Thus, $G$ contains a  directed path of length at least $m-1$.  \end{proof}

\subsection{Kneser graphs and their complements}

\begin{theorem}[\cite{L78}]\label{chr-num-Kneser} 
For $n\geq 2k-1$, the chromatic number of the Kneser graph $K(n,k)$ is $n-2k+2$. \end{theorem} 

The following is a well known and easy to see fact.

\begin{lemma}\label{clique-lemma} When $n < ck$, $K(n,k)$ does not contain cliques $K_c$ of size $c$, whereas it does contain such cliques when $n \geq ck$.  \end{lemma}

An {\em independent set} is a set of vertices  no two of which are connected by an edge.  The {\em independence number} of a graph is the size of a maximal independent set.
 
\begin{theorem}[The Erd\H{o}s-Ko-Rado Theorem; \cite{EKR61}]\label{independence-number-Kneser} The independence number of the Kneser graph $K(n,k)$, equivalently, the size of the largest clique in $\overline{K(n,k)}$, is ${n-1\choose k-1}$.\end{theorem}

\begin{theorem}[\cite{Baranyai}]\label{chromatic-numb-complement} The chromatic number of the graph $\overline{K(n,k)}$ is $\left\lceil {n\choose k}/\left\lfloor \frac{n}{k} \right\rfloor \right\rceil$. \end{theorem}

\subsection{Semi-transitive graphs and word-representability}

\begin{theorem}[\cite{HKP16}]\label{3-color-graphs}  Any $3$-colourable graph is semi-transitive. \end{theorem}

\begin{proof} Colour the vertices in a given 3-colourable graph $G$ in colours 1, 2 and 3, and orient the edges from a smaller colour to a larger colour. Such orientation is clearly acyclic. Moreover, it is shortcut-free since the longest direct path is of length 2, while for a shortcut we need a directed path of length 3. Thus, the orientation is semi-transitive. \end{proof}

\begin{lemma}[\cite{AKM15}]\label{lemma} Suppose that the vertices in $\{a,b,c,d\}$ induce a subgraph $S$ in a partially oriented graph such that $a\rightarrow b$ and $b\rightarrow c$ are edges, $cd$ and $da$ are non-oriented edges, and $S$ is different from the complete graph $K_4$. Then, the unique way to orient $cd$ and $da$ in order not to create a directed cycle or a shortcut is $a\rightarrow d$ and $d\rightarrow c$. \end{lemma}

\begin{proof} Indeed, suppose that the edge $cd$ is oriented as $c\rightarrow d$. Then, orienting $ad$ will either give the cycle $a\rightarrow b\rightarrow c\rightarrow d\rightarrow a$, or the shortcut with the shortcutting edge $a\rightarrow d$. Thus, the orientation of $cd$ must be $d\rightarrow c$. To complete the proof, we note that orienting $da$ as $d\rightarrow a$ will give a shortcut with the shortcutting edge $d\rightarrow c$.\end{proof}

To accomodate a simple proof of Theorem~\ref{semi-trans-complement}, next we introduce the notion of a word-representable graph  and state the relation between semi-transitive graphs and word-representable graphs in Theorem~\ref{equiv-sem-word}.

Letters $x$ and $y$ alternate in a word $w$ if after deleting in $w$ all letters but the copies of $x$ and $y$ we either obtain a word $xyxy\cdots$ (of even or odd length) or a word $yxyx\cdots$ (of even or odd length). For example, the letters 2 and 5 alternate in the word 11245431252, while the letters 2 and 4 do not alternate in this word. A simple graph $G=(V,E)$ is {\em word-representable} if there exists a word $w$ over the alphabet $V$ such that letters $x$ and $y$ alternate in $w$ iff $xy\in E$. By definition, $w$ {\em must} contain {\em each} letter in $V$. We say that $w$ {\em represents} $G$. For example, each complete graph $K_n$ can be represented by any permutation $\pi$ of $\{1,2,\ldots,n\}$, or by $\pi$ concatenated any number of times.   Also, the empty graph $E_n$ (also known as an edgeless graph, or a null graph) on vertices $\{1,2,\ldots,n\}$ can be represented by $1122\cdots nn$, or by any permutation concatenated with the same permutation written in the reverse order.

\begin{theorem}[\cite{HKP16}]\label{equiv-sem-word} A graph is semi-transitive if and only if it is word-representable.\end{theorem}

\section{Semi-transitivity of Kneser graphs}\label{sec3}

\begin{theorem}\label{3-col-Kneser} For $k\leq n\leq 2k+1$, $K(n,k)$ is semi-transitive. \end{theorem}

\begin{proof} By Theorem~\ref{chr-num-Kneser}, for $n\leq 2k+1$, $K(n,k)$ is 3-colourable, and thus is semi-transitive by Theorem~\ref{3-color-graphs}. \end{proof}

The following lemma is easy to see. 

\begin{lemma}\label{embedding-lem} The graph $K(n,k)$ (resp., $\overline{K(n,k)}$) is an induced subgraph in any graph $K(m,k)$ (resp., $\overline{K(m,k)}$) for $m\geq n$.\end{lemma}

\begin{proof} The subgraph of $K(m,k)$ induced only by the vertices formed by the elements in $\{1,2,\ldots,n\}$  is isomorphic to $K(n,k)$. The statement for the complements now follows as well. \end{proof}

In the following theorem, we repeatedly use the fact that $K(6,2)$ is $K_4$-free by Lemma~\ref{clique-lemma}, which allows the application of Lemma~\ref{lemma}. 

\begin{theorem}\label{Knes-K62-thm} The Kneser graph $K(n,2)$ is not semi-transitive for $n\geq 6$. \end{theorem}

\begin{proof} By Lemma~\ref{embedding-lem}, and the hereditary nature of semi-transitivity, it is sufficient to prove the theorem for  $K(6,2)$. 

We proceed by contradiction. Assume that $K(6,2)$ can be oriented semi-transitively and fix such an orientation. Since the chromatic number of $K(6,2)$ is 4 by Theorem~\ref{chr-num-Kneser}, by Theorem~\ref{long-directed-path}  the oriented copy of $K(6,2)$ must contain a directed path $A\rightarrow B\rightarrow C \rightarrow D$. Note that if the edge $A\rightarrow D$ exists, we would obtain a contradiction, since the vertices $A,B,C,D$ would induce a shortcut ($K(6,2)$ is $K_4$-free). Thus, there is no edge in $K(6,2)$ between $A$ and $D$, and without loss of generality, we can assume that one of the following three cases occurs, where $abcdef$ is a permutation of $[6]$: 
\begin{itemize}
\item[Case 1.] $A=ab$, $B=cd$, $C=be$ and $D=af$ ($C$ is uniquely determined once the assumption that $D$ involves $a$ is made). Note that in this case, $B\rightarrow D$ is an edge. 
\item[Case 2.] $A=ab$, $B=cd$, $C=ef$ and $D=ac$ (the assumption here is that $D$ shares an element with $A$ and $B$). This case is equivalent to Case 1, since reversing all edges in a semi-transitive orientation gives a semi-transitive orientation, and the letters $a,b,\ldots$ can be renamed. So, Case 2 does not need to be considered. 
\item[Case 3.] $A=ab$, $B=cd$, $C=be$ and $D=ac$ (the assumption here is that $D$ shares an element with $A$ and $B$, and $C$ shares an element with $A$). In this case, consider the 4-cycle induced by $A$, $B$, $C$ and $df$. By Lemma~\ref{lemma}, we must have the following edges: $A\rightarrow df$ and $df \rightarrow C$. But then, the directed path $A\rightarrow df \rightarrow C \rightarrow D$ is equivalent to the path in Case 1, so Case 3 does not need to be considered. 
\end{itemize}

Thus, we only need to consider Case 1 and arrive at a contradiction. Our strategy here is to consider a number of graphs induced by 4 vertices (in certain order) in which orientation of edges is uniquely determined from our assumptions. Eventually, we will show that shortcuts are unavoidable. In what follows, for convenience, we do not use the letters $A$, $B$, $C$ and $D$, writing the 2-set partitions corresponding to them instead. 

\begin{itemize}
\item From the graph induced by $ab$, $df$, $cd$, $be$, by Lemma~\ref{lemma}, we must have $ab\rightarrow df$ and $df \rightarrow  be$.
\item If $ce \rightarrow ab$ is an edge, then either we have the cycle $ce \rightarrow  ab \rightarrow  cd \rightarrow  be \rightarrow  af \rightarrow  ce$, or we have the shortcut  with the shortcutting edge $ce\rightarrow af$ and the long path $ce\rightarrow  ab\rightarrow  cd\rightarrow  be \rightarrow  af$. Thus, we must have $ab\rightarrow ce$.
\item If $af\rightarrow ce$ is an edge, then $ab\rightarrow cd\rightarrow be\rightarrow af\rightarrow ce$ is the long path and $ab\rightarrow ce$ is the shortcutting edge in a shortcut. Thus, we must have $ce\rightarrow af$. 
\item Replacing $ce$ by $de$ in the last two bullet points, we see that $ab\rightarrow de$ and $de\rightarrow af$ are edges. 
\item If $ce\rightarrow df$ is an edge, then $ce\rightarrow df\rightarrow be\rightarrow af$ is the long path and $ce\rightarrow af$ is the shortcutting edge in a shortcut. Thus, we must have $df\rightarrow ce$. 
\item From the graph induced by $df$, $be$, $af$ and $bc$, by Lemma~\ref{lemma}, we must have $df\rightarrow bc$ and $bc\rightarrow af$. 
\item From the graph induced by $ab$, $df$, $de$ and $bc$, to avoid a shortcut, we must have $de\rightarrow bc$.
\item From the graph induced by $ab$, $de$, $ef$ and $bc$, by Lemma~\ref{lemma}, we must have $ab\rightarrow ef$ and $ef\rightarrow bc$.
\item From the graph induced by $ef$, $bc$, $bd$ and $af$, by Lemma~\ref{lemma}, we must have $ef\rightarrow bd$ and $bd\rightarrow af$.
\item From the graph induced by $ab$, $ef$, $ce$ and $bd$, by Lemma~\ref{lemma}, we must have $ab\rightarrow ce$ and $ce\rightarrow bd$.
\item From the graph induced by $df$, $ce$, $ae$ and $bd$, by Lemma~\ref{lemma}, we must have $df\rightarrow ae$ and $ae\rightarrow bd$.
\item From the graph induced by $cd$, $ae$, $be$ and $df$, we must have $cd\rightarrow ae$ to avoid a shortcut.
\item From the graph induced by $ab$, $cf$, $cd$ and $ae$, we must have $cf\rightarrow ae$ to avoid a shortcut.
\item From the graph induced by $cf$, $ae$, $de$ and $bc$, we must have $ae\rightarrow bc$ to avoid a shortcut.
\end{itemize}
But we obtain a contradiction, since there is a shortcut with the long path $cd\rightarrow ae\rightarrow bc\rightarrow af$ and the shortcutting edge $cd\rightarrow af$. Thus, $K(6,2)$ is not semi-transitively orientable.
\end{proof}

\begin{remark}{\em We note that $K(6,2)$, and thus any $K(m,2)$ for $m\geq 6$, is not a minimal non-semi-transitive graph. Software check (using \cite{G}) shows that removing any vertex  in $K(6,2)$ gives a non-semi-transitive graph.}
 \end{remark}
 
The following theorem generalizes Theorem~\ref{Knes-K62-thm}. 
 
\begin{theorem}\label{Knes-main-thm} For $n\geq 15k-24$ and $k\geq 2$, $K(n,k)$ is not semi-transitive. 
\end{theorem}

\begin{proof} We claim that such a $K(n,k)$ contains $K(6,2)$ as an induced subgraph, and thus is non-semi-transitive by Theorem~\ref{Knes-K62-thm}. Indeed, consider inserting  to each 2-subset involved in building $K(6,2)$ $(k-2)$ distinct elements so that no two 2-subsets receive the same new element. Then, the number of new elements is $15(k-2)$, and the total number of elements is $6 + 15(k-2)=15k-24$, where 6 came from the 6 elements used to build $K(6,2)$. Our construction shows that $K(6,2)$ is an induced subgraph in $K(15k-24,k)$ since no edge in $K(6,2)$ is affected by the construction, and thus in any $K(n,k)$,  $n\geq 15k-24$, by Lemma~\ref{embedding-lem}. Since $K(6,2)$ is non-semi-transitive by Theorem~\ref{Knes-K62-thm}, $K(n,k)$ is non-semi-transitive for   $n\geq 15k-24$.\end{proof}

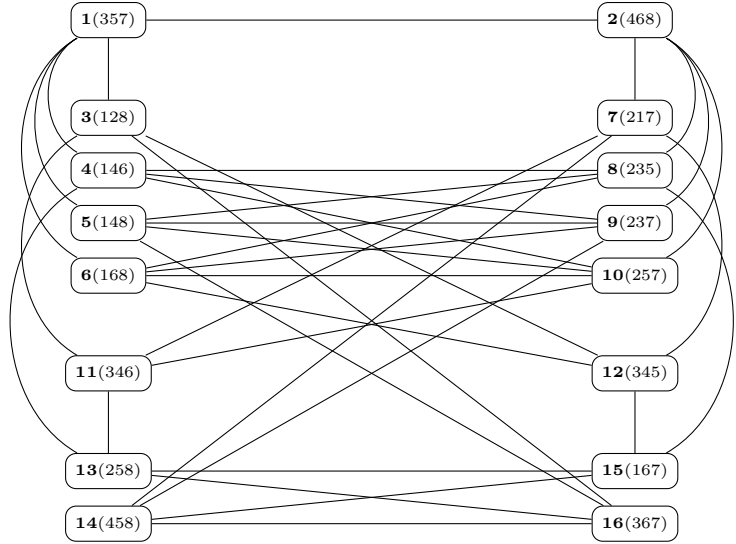
\begin{figure}
\begin{center}
\begin{tikzpicture}[node distance=1cm,auto,main node/.style={rectangle,rounded corners,draw,align=center}]

\node[main node] (1) {{\tiny {\bf 1}(357)}};
\node[main node] (2) [right of=1,xshift=6cm] {{\tiny {\bf 2}(468)}};

\node[main node] (3) [below of=1,yshift=-0.3cm] {{\tiny {\bf 3}(128)}};
\node[main node] (4) [below of=3,yshift=0.3cm] {{\tiny {\bf 4}(146)}};
\node[main node] (5) [below of=4,yshift=0.3cm] {{\tiny {\bf 5}(148)}};
\node[main node] (6) [below of=5,yshift=0.3cm] {{\tiny {\bf 6}(168)}};

\node[main node] (7) [below of=2,yshift=-0.3cm] {{\tiny {\bf 7}(217)}};
\node[main node] (8) [below of=7,yshift=0.3cm] {{\tiny {\bf 8}(235)}};
\node[main node] (9) [below of=8,yshift=0.3cm] {{\tiny {\bf 9}(237)}};
\node[main node] (10) [below of=9,yshift=0.3cm] {{\tiny {\bf 10}(257)}};

\node[main node] (11) [below of=6,yshift=-0.3cm] {{\tiny {\bf 11}(346)}};

\node[main node] (13) [below of=11,yshift=-0.3cm] {{\tiny {\bf 13}(258)}};
\node[main node] (14) [below of=13,yshift=0.3cm] {{\tiny {\bf 14}(458)}};

\node[main node] (12) [below of=10,yshift=-0.3cm] {{\tiny {\bf 12}(345)}};

\node[main node] (15) [below of=12,yshift=-0.3cm] {{\tiny {\bf 15}(167)}};
\node[main node] (16) [below of=15,yshift=0.3cm] {{\tiny {\bf 16}(367)}};

\path
(1) edge (2)
      edge (3)
      edge  [bend right=60] (4)
      edge [bend right=60] (5)
      edge [bend right=60] (6);
\path
(2) edge (7)
      edge  [bend left=60] (8)
      edge [bend left=60] (9)
      edge [bend left=60] (10);
\path
(3) edge [bend right=60] (11) edge (12) edge (16);
\path
(7) edge [bend left=60] (12) edge (11) edge (14);
\path
(4) edge (8) edge (9) edge (10) edge [bend right=60] (13);
\path
(5) edge (8) edge (9) edge (10) edge (16);
\path
(6) edge (8) edge (9) edge (10) edge (12);
\path
(8) edge [bend left=60] (15);
\path
(9) edge (14);
\path
(10) edge (11);
\path
(13) edge (15) edge (16) edge (11);
\path
(14) edge (15) edge (16);
\path
(12) edge (15);

\end{tikzpicture}
\caption{A minimal non-semi-transitive subgraph $S$ of $K(8,3)$. Name of a vertex is in bold, and the set partition corresponding to it is in parenthesis.}\label{K83-fig}
\end{center}
\end{figure}

To extend our knowledge on semi-transitivity of Kneser graphs to the unknown cases, we have looked at $K(8,3)$ having 56 vertices and 280 edges. Using the software \cite{G}, we have learned that the subgraph of $K(8,3)$ formed by the 46 lexicographically smallest vertices (123, 124, 126, etc) is semi-transitive, and a semi-transitive orientation was found within 2 seconds. However, adding one more vertex to the subgraph (456, the 47th lexicographically smallest one), resulted in no result returned by the software within a few hours, which was an indication, but not a given fact, that the graph may not be semi-transitive. Thus, our next goal was to find a non-semi-transitive subgraph $S$ of $K(8,3)$ of a smaller size, for which the software would return a definite answer on non-semi-transitivity of $S$, and thus of $K(8,3)$. Such a graph $S$, presented in Figure~\ref{K83-fig}, was found using clustering nodes into independent sets and then eliminating certain nodes. Checking non-semi-transitivity of $S$ takes just below 4 seconds using the software  \cite{G}, and the minimality of $S$ is straightforward to check using the same software. $S$ contains 16 vertices and 36 edges. 

Of course, it would be desirable to find a non-computer based proof of non-semi-transitivity of $K(8,3)$, e.g. similar to that of non-semi-transitivity of $K(6,2)$ in Theorem~\ref{Knes-K62-thm}, but we were not able to achieve it. We note that determining if a triangle-free graph is semi-transitive is an NP-hard problem \cite{HKP16}, and presenting all $2^{35}$ orientations (one edge can be assumed to have any fixed orientation) and showing a shortcut, or cycle, in each of them is not feasible for a human. In fact, there is the branching method explained in Section 4.5 in \cite{KL15} (also, see \cite{CKL17}) to dramatically decrease the number of cases to consider while proving that a graph is not semi-transitive. The basic idea of the method is to avoid branching for those edges for which orientation is uniquely determined in order to be semi-transitive. However, this method, being efficient for not so many edges, still leads to too many cases to consider for $S$, and thus is rather useless in the situation.     

\begin{remark}{\em $S$ in Figure~\ref{K83-fig} is the first explicit example of a triangle-free non-semi-transitive graph. The existence of such graphs was established in \cite{HKP11} using Erd\H{o}s' theorem (also see \cite[Section 4.4]{KL15}).}\end{remark}

\section{Semi-transitivity of the complement of Kneser graphs}

\begin{theorem}\label{semi-trans-complement} For $n\leq 2k$, $\overline{K(n,k)}$ is semi-transitive. \end{theorem}

\begin{proof} Clearly, if $n<2k$, then $\overline{K(n,k)}$ is a complete graph (no pair of $k$-subsets is disjoint), and thus it is semi-transitive. On the other hand, if $n=2k$, then $\overline{K(n,k)}$ is a complete graph with a perfect matching removed (each non-edge is formed by a $k$-subset and its complement). Label  $\overline{K(2k,k)}$ so that the non-edges are formed by the vertices $2i-1$ and $2i$ for $1\leq i\leq k$ and let $x={n\choose k}$. Then, the word $1234\cdots (x-1)x2143\cdots x(x-1)$ represents $\overline{K(2k,k)}$, and by Theorem~\ref{equiv-sem-word},  $\overline{K(2k,k)}$ is semi-transitive.
 \end{proof}

\begin{lemma}\label{inequality} For $k\geq 4$, we have ${2k\choose k-1}+k<\frac{1}{2}{2k+1 \choose k}-2$. \end{lemma}

\begin{proof} Using ${2k+1 \choose k}={2k \choose k}+{2k \choose k-1}$, we need to prove that
$$\frac{1}{2}{2k\choose k-1}+k< \frac{1}{2}{2k \choose k}-2, \mbox{ or }$$ 
$$\frac{(2k)!}{(k-1)!(k+1)!}+2k< \frac{(2k)!}{k!k!}-4, \mbox{ or }$$
$$k(2k)!+2kk!(k+1)!<(k+1)(2k)!-4k!(k+1)! \mbox{ or }$$
$$(2k+4)k!(k+1)!<(2k)!.$$ The last statement can be proved by induction on $k$ with the easy to check base case of $k=4$. Indeed, using the induction hypothesis, we have $$(2(k+1))!=(2k+2)(2k+1)(2k)! > (2k+2)(2k+1)(2k+4)k!(k+1)!$$
$$> (2k+6)(k+1)!(k+2)!$$ where the last inequality follows from the easy to see, for $k\geq 1$, inequality $$(2k+2)(2k+1)(2k+4)>(2k+6)(k+2)(k+1),$$
or $6k^3+16k^2+6k-4>0$. \end{proof}

\begin{theorem}\label{n=2k+1} For $k\geq 2$, the graph $\overline{K(2k+1,k)}$ is not semi-transitive. \end{theorem}

\begin{proof} For $k=2$, we note that $\overline{K(5,2)}$ is the line graph of $K_5$, and it is proved in~\cite{KSSU} to be non-word-representable, and thus,  $\overline{K(5,2)}$ is not semi-transitive by Theorem~\ref{equiv-sem-word}.

Let  $k=3$, and suppose that $\overline{K(7,3)}$ admits a semi-transitive orientation. Fix such an orientation. By Theorem~\ref{chromatic-numb-complement}, the chromatic number of $\overline{K(7,3)}$ is 18, and thus, by Theorem~\ref{long-directed-path}, $\overline{K(7,3)}$ contains a directed path $X_1\rightarrow X_2 \rightarrow \cdots \rightarrow X_{18}$. Moreover, by Theorem~\ref{independence-number-Kneser}, the largest clique in $\overline{K(7,3)}$ is of size 15, and thus if $X_1\rightarrow X_i$ is an edge for $i\in \{16, 17, 18\}$, $X_1\rightarrow X_i$ would be the  shortcutting edge for the long path $X_1\rightarrow X_2 \rightarrow \cdots \rightarrow X_{i}$ (the graph induced by $X_1$, $X_2,\ldots,X_i$ cannot be transitive as that would mean that $\overline{K(7,3)}$ has a clique of size $>15$). Therefore, $X_1$ is not connected to $X_{16}$, $X_{17}$ and $X_{18}$, so that if $X_1=123$ (without loss of generality), then $X_{16}$, $X_{17}$ and $X_{18}$ are formed using the elements in $\{4,5,6,7\}$. Without loss of generality, assume that $X_2$ involves the element 4. Since at least one of $X_{17}$ and $X_{18}$ must involve 4, say $X_m$, where $m\in\{17,18\}$,  $X_2\rightarrow X_m$ must be the shortcutting edge with the long path  $X_2\rightarrow X_3 \rightarrow \cdots \rightarrow X_{m}$. Contradiction. Thus, $\overline{K(7,3)}$ is not semi-transitive.

Finally, let $k\geq 4$ and suppose that $K(2k+1,k)$ admits a semi-transitive orientation. Fix such an orientation. By Theorem~\ref{chromatic-numb-complement}, the chromatic number of $\overline{K(2k+1,k)}$ is $t=\left\lceil \frac{1}{2}{2k+1\choose k}\right\rceil$ and thus, by Theorem~\ref{long-directed-path}, $\overline{K(2k+1,k)}$ contains a directed path $X_1\rightarrow X_2 \rightarrow \cdots \rightarrow X_{t}$. Moreover, by Theorem~\ref{independence-number-Kneser}, the largest clique in $\overline{K(2k+1,k)}$ is of size $s={2k\choose k-1}$. Thus, $X_1$ cannot be connected to $X_i$ for $s+1\leq i\leq t$, and the number of such $X_i$s cannot exceed ${(2k+1)-k\choose k}=k+1$ (the $k$ elements used in $X_1$ are not available for $X_i$s). But then, we must have the following inequality
$$t-(s+1)+1\leq  k+1 \Rightarrow \left\lceil \frac{1}{2}{2k+1\choose k}\right\rceil - {2k\choose k-1}\leq  k+1 \Rightarrow $$
$$\frac{1}{2}{2k+1\choose k}-1- {2k\choose k-1}\leq  k+1,$$ which contradicts Lemma~\ref{inequality}. Thus, $K(2k+1,k)$ is not semi-transitively orientable for $k\geq 4$.
 \end{proof}

As an immediate corollary to Theorems~\ref{semi-trans-complement} and~\ref{n=2k+1} and Lemma~\ref{embedding-lem}, we have the following result. 

\begin{theorem}\label{complement-main} The complement graph $\overline{K(n,k)}$ of $K(n,k)$ is semi-transitive if and only if $n\leq 2k$. \end{theorem}

\begin{proof} By Theorem~\ref{semi-trans-complement}, $\overline{K(n,k)}$ is semi-transitive if $n\leq 2k$. Now, suppose that $n>2k$. Since  $\overline{K(2k+1,k)}$ is an induced subgraph in $\overline{K(n,k)}$ by Lemma~\ref{embedding-lem}, and  $\overline{K(2k+1,k)}$ is non-semi-transitive by Theorem~\ref{n=2k+1}, $\overline{K(n,k)}$ is also non-semi-transitive. \end{proof}

\section{Concluding remarks}
In this paper, we show that for $n\geq 15k-24$, $K(n,k)$ is not semi-transitive, while for $k\leq n\leq 2k+1$, $K(n,k)$ is semi-transitive. Also, we have used computations to show that the triangle-free graph $K(8,3)$ is not semi-transitive. Moreover, we have completely characterized semi-transitivity of the complement graph $\overline{K(n,k)}$ by showing that $\overline{K(n,k)}$ is semi-transitive if and only if $n\geq 2k$. We conclude the paper with the following open problems.

\begin{itemize}
\item Give a non-computer based proof of non-semi-transitivity of $K(8,3)$. 
\item More generally, is $K(2k+2,k)$ non-semi-transitive for any $k\geq 3$? If that would be the case, then  we would complete the classification of semi-transitive Kneser graphs by Lemma~\ref{embedding-lem} and Theorems~\ref{Knes-K62-thm} and~\ref{Knes-main-thm}, namely, that would imply that a Kneser graph is semi-transitive if and only if it is 3-colourable. 
\item Are there any smaller triangle-free non-semi-transitive graphs (by the number of vertices and/or the number of edges) than the graph $S$ in Figure~\ref{K83-fig}? Possible candidates for such a graph could be some subgraphs of $K(8,3)$, which could then also help to prove rigorously non-semi-transitivity of $K(8,3)$ (the fewer than 36 edges in such a graph could possibly be handled by the branching method in \cite[Section 4.5]{KL15}). As a relevant observation to searching for candidates here, note that triangle-free planar graphs are always semi-transitive as they are 3-colorable \cite{HKP11}.  
\end{itemize}

\section*{Acknowledgment} We are grateful to  Li-Da Tong for raising our interest in Kneser graphs.

\end{document}